\documentclass[11pt]{article}
\usepackage{fullpage} 
\usepackage{authblk}
\usepackage{ucs}
\usepackage{amssymb}
\usepackage{amsthm}
\usepackage{amsmath}
\usepackage{latexsym}
\usepackage[cp1251]{inputenc}
\usepackage[english]{babel}
\usepackage{graphicx}
\usepackage{wrapfig}
\usepackage{caption}
\usepackage{subcaption}
\usepackage{txfonts}
\usepackage{lmodern}
\usepackage{mathrsfs}
\usepackage{algorithm}
\usepackage{dsfont}
\usepackage{enumerate}
\usepackage{hyperref}
\usepackage{multicol}
\usepackage{csquotes}
\usepackage{stmaryrd}
\usepackage{tikz}
\usepackage{comment}
\usepackage{enumitem}
\usepackage[initials]{amsrefs} 
\usepackage{cite}

\newtheorem{theo}{Theorem}

\newtheorem{obs}[theo]{Observation}
\newtheorem{lemma}[theo]{Lemma}

\newtheorem{defn}[theo]{Definition}

\numberwithin{theo}{section}

 \allowdisplaybreaks

\author{}
\title{Applications of Sparse Hypergraph Colorings}
\date{\today}

\author{Felix Christian Clemen \\
Karlsruhe Institute of Technology, 76133 Karlsruhe, Germany}

\begin{document}
\maketitle
\begin{abstract}
Many problems in extremal combinatorics can be reduced to determining the independence number of a specific auxiliary hypergraph. We present two such problems, one from discrete geometry and one from hypergraph Tur\'an theory. Using results on hypergraph colorings by Cooper-Mubayi and Li-Postle, we demonstrate that for those two problems the trivial lower bound on the independence number can be improved upon:

\begin{itemize}
\item
Erd\H{o}s, Graham, Ruzsa and Taylor asked to determine the largest size, denoted by $g(n)$, of a subset $P$ of the grid $[n]^2$ such that every pair of points in $P$ span a different slope. Improving on a lower bound by Zhang from 1993, we show that $$g(n)=\Omega \left(  \frac{n^{2/3} (\log \log n)^{1/3} }{ \log^{1/3}n} \right).$$ 
\item
Let $H^r_3$ denote an $r$-graph with $r+1$ vertices and $3$ edges. Recently, Sidorenko proved the following lower bounds for the Tur\'an density of this $r$-graph: $\pi(H^r_3)\geq r^{-2}$ for every $r$, and $\pi(H^r_3)\geq (1.7215 - o(1)) r^{-2}$. We present an improved asymptotic bound: $$\pi(H^r_3)=\Omega\left(r^{-2}  \log^{1/2} r  \right).$$ 
\end{itemize}
\end{abstract}

\section{Introduction}
A \emph{hypergraph} is a pair $(V,E)$ where $V$ is a set whose elements are called \emph{vertices}, and $E$ is a family of subsets of $V$, called \emph{edges}. A hypergraph is a \emph{$k$-graph} if every edge has size $k$.  
A hypergraph has \emph{rank $k$} if every edge contains at least $2$ and at most $k$ vertices.

A subset of vertices $I\subseteq V$ is called an \emph{independent set} in $H=(V,E)$ if there is no edge $e\in E$ contained in $I$. The \emph{independence number} of a hypergraph $H$, denoted by $\alpha(H)$, is the size of a maximum independent set. Many problems in extremal combinatorics are related to the problem of determining the independence number of a specific auxiliary hypergraph, see e.g. \cite{ajtai, Campos,Fer, Furedi, Krivelevich}. In this paper we present two such problems, one from discrete geometry and one from hypergraph Tur\'an theory. For both we use probabilistic tools to present an improved bound on the independence number. 

\subsection{Application 1: The distinct slopes in the grid problem}
The unit distance problem, a classical problem in combinatorial geometry, posed by Erd\H{o}s\cite{MR15796} in 1946, asks to determine the maximum number of unit segments in a set of $n$ points in the plane. Since then a great variety of extremal problems in finite planar point sets have been studied. Here, we look at such a question concerning the \emph{grid} $[n]^2$. 

A subset $P\subseteq [n]^2$ of the grid has \emph{distinct slopes} if there are no two pairs of points spanning lines with the same slope. Note that such a set $P$ does not contain collinear triples, nor 4 points forming a trapezoid. Denote by $g(n)$ the size of the largest subset $P$ of the grid $[n]^2$ with distinct slopes. In 1992, Erd\H{o}s, Graham, Ruzsa and Taylor~\cite{EGRT} asked to determine this function and proved 
$$  \Omega(n^{1/2})=g(n)=O(n^{4/5}).$$
Their lower bound comes from an algebraic construction. In 1993, Zhang~\cite{zhang} improved the lower bound by a deletion method argument resulting in $$g(n)
=\Omega\left(\frac{n^{2/3}}{\log^{1/3}n}\right).$$
To do so, Zhang~\cite{zhang} observed that the number $T$ of trapezoids in $[n]^2$ is $O(n^6 \log n)$ and the number $C$ of collinear triples is $O(n^4\log n)$. Therefore, when considering a $p$-random subset of the grid, for $p=c n^{-4/3}\log^{-1/3} n$, where $c$ is sufficiently small it holds that
$$pn^2-p^4T-p^3C=\Omega\left(\frac{n^{2/3}}{\log^{1/3}n}\right).$$
Thus, there exists a subset $P\subseteq [n]^2$ of the desired size with no trapezoids and no collinear triples, and therefore $P$ has distinct slopes. Here, we present the first improvement to this deletion method argument in 30 years.
\begin{theo}
\label{pointmain}
There exists a subset of the grid $[n]^2$ with distinct slopes and of size
$$g(n)=\Omega \left(  \frac{n^{2/3} (\log \log n)^{1/3} }{ \log^{1/3}n} \right).$$

\end{theo}

\subsection{\texorpdfstring{Application 2: The Tur\'an density of the $r$-graph with $3$ edges on $r+1$ vertices}{TEXT}}

Given an $r$-graph $H$, the \emph{Tur\'an function} $\textup{ex}(n,H)$ is the maximum number of edges in an $H$-free $n$-vertex $r$-graph. The \emph{Tur\'an density} of $H$, denote by $\pi(H)$, is 
 $$
 \pi(H)=\lim_{n\to \infty} \frac{\textup{ex}(n,H)}{ \binom{n}{r}}.$$
 It is a central question in extremal combinatorics to determine $\pi(H)$ for non $r$-partite $r$-graphs when $r\geq 3$. The Tur\'an density of any $r$-graph with at most 2 edges is $0$. A natural problem is to consider $r$-graphs with $3$ edges. Denote by $H^r_3$ an $r$-graph with 3 edges on $r+1$ vertices. Note that all such $r$-graphs are isomorphic. 
 
 It can be observed quickly that $\pi(H^r_3) \rightarrow 0$, see e.g. \cite{BLM} by Bollob\'as, Leader and Malvenuto. Specifically, $r!/r^r\leq \pi(H^r_3)\leq 2/(r+1)$. The lower bound holds, because $H^r_3$ is not $r$-partite. The upper bound holds because in an $H^r_3$-free $r$-graph every $r+1$ vertices span at most $2$ edges. Frankl and F\"uredi~\cite{FF} proved  $\pi(H^r_3)\geq 2^{1-r}$ using a geometric construction. Recently, this result was improved by Sidorenko~\cite{sid} to $\pi(H^r_3)\geq r^{-2}$ for all $r$ and $\pi(H^r_3)\geq (1.7155-o(1))r^{-2}$ for $r \rightarrow \infty$. The best general upper bound is $\pi(H^r_3)\leq 1/r$ by Gunderson and Semerano~\cite{gund2} who used a double counting argument. See \cite{gunderson2022tur} for more precise bounds on $\pi(H^r_3)$ for small values of $r$. We present an asymptotic improvement on the lower bound. 
\begin{theo}
\label{greedyK3}
$\pi(H^r_3)=\Omega\left(\frac{\sqrt{\log r}}{r^2}\right)$.
\end{theo}

\subsection{Sparse hypergraph coloring}
In this subsection we present the tools we are using to prove Theorems~\ref{pointmain} and \ref{greedyK3}.

 A \emph{proper coloring} of a hypergraph $H$ is an assignment of colors to the vertices so that no edge is monochromatic. The smallest number of colors such that a proper coloring of $H$ with this many colors exists, is called the \emph{chromatic number} of $H$ and denoted by $\chi(H)$.

Let $H$ be a rank $k$ hypergraph. For an integer $2 \leq l \leq k$, and a set $S$ of vertices, $1 \leq |S| < \ell$, we define $\deg_\ell(S,H)$ to be the number of size $\ell$ edges containing the set $S$. When $H$ is clear from context, we omit $H$ and write $\deg_\ell(S)$ instead. The maximum $l$-degree of $H$, denoted by $\Delta_\ell(H)$, is the maximum of $\deg_{\ell}(\{v\},H)$ over all vertices $v\in V(H)$. The maximum $(s,\ell)$-codegree of $H$, denoted by $\Delta_{s,\ell}(H)$, is the maximum of $\deg_{\ell}(S,H)$ over all vertex sets $S$ of size $s$ in $H$. For a $k$-graph, we write $\Delta(H)$ for the \emph{maximum degree} $\Delta_{1,k}(H)$. Again, when $H$ is clear from context, we omit $H$ and simply write $\Delta_{\ell}$, $\Delta_{s,\ell}$ and $\Delta$ respectively.

An elementary greedy coloring algorithm shows that any graph $G$ has chromatic number $\chi(G) \leq \Delta_2 + 1$. In 1996, Johansson~\cite{johansson} proved that $\chi(G) = O(\Delta_2/ \log \Delta_2)$ for any triangle-free graph $G$. Corresponding bounds on the independence number were established by Ajtai- Koml{\'o}s-Szemer{\'e}di~\cite{ajtai} and Shearer~\cite{shearer}. Frieze and Mubayi~\cite{FM}
generalized Johansson's theorem to all $3$-uniform linear hypergraphs, and
Cooper and Mubayi~\cite{CM} later generalized it to all
rank 3 hypergraphs. The following result by Cooper and Mubayi~\cite{MJ} gives an upper bound on the chromatic number $\chi(G)$  of a $k$-graph under a codegree condition. 

\begin{theo}[Cooper, Mubayi~\cite{MJ}]
\label{Mub}
Fix $k \geq 3$. Let $H$ be a $k$-graph. If
\begin{align}
\label{condMub}
\Delta_{\ell,k}(H) \leq \Delta^{\frac{k-\ell}{k-1}}/f \quad \quad \text{for} \ \ell=2,\ldots,k-1,
\end{align}
then
$$\chi(H)=O\left(\left(\frac{\Delta}{\log f}\right)^{1/(k-1)} \right).$$
\end{theo}

Recently, Li and Postle~\cite{linali} generalized Theorem~\ref{Mub} to hold for rank $k$ hypergraphs. Given two hypergraphs $F_1$ and $F_2$, a map $\Phi : V(F_1) \rightarrow V(F_2)$ is an \emph{isomorphism} if for all subsets
$E \subseteq V(F_1)$, we have $\Phi(E) \in F_2$ if and only if $E \in F_1$. If there exists an isomorphism $\Phi : V(F_1) \rightarrow V(F_2)$,
we say $F_1$ is isomorphic to $F_2$ and denoted it by $F_1 \cong_\Phi F_2$. For two hypergraphs $F, H$ and a vertex
$v \in V(F)$, let
$$\Delta_{F,v}(H) = \max_{u\in V (H)}
|{F' \subseteq H : F' \cong_\Phi F \text{ and } \Phi(u) = v}|$$
and
$$\Delta_F (H) = \min_{v\in V(F)} \Delta_{F,v}(H).$$

\begin{theo}[Li, Postle~\cite{linali}]
\label{LPspar}
Fix $k\geq 3$. Let $\mathcal{H}$ be a rank $k$ hypergraph.
Suppose that for all $2\leq s<\ell \leq k$, 
\begin{align}
\Delta_{s, \ell}(\mathcal{H})\leq \left(\max_{2\leq\ell\leq k}\Delta_{\ell}(\mathcal{H})^{1/(\ell-1)}\right)^{\ell-s}/f, 
\label{codegreecon}
\end{align}
and additionally for the graph triangle $K_3$,
\begin{align}
\Delta_{K_3}(\mathcal{H})\leq \left(\max_{2\leq\ell\leq k}\Delta_{\ell}(\mathcal{H})^{1/(\ell-1)}\right)^{2}/f.
\label{codegreeconK3}
\end{align}
Then we have
\[
\chi(\mathcal{H})\leq O\left(\max_{2\leq \ell \leq k} \left\{\left(  \frac{\Delta_{\ell}(\mathcal{H})}{\log f} \right)^{\frac{1}{\ell-1}} \right\}\right).
\]
\end{theo}
We point out that Li and Postle conclude Theorem~\ref{LPspar} from a more general result (see Theorem~1.7 in \cite{linali}).

This paper is organized as follows. In Section~\ref{sec: grid} we first collect several observations about the grid, and then we prove Theorem~\ref{pointmain} using Theorem~\ref{LPspar}. Finally, in Section~\ref{Turan}, we present the proof of Theorem~\ref{greedyK3} using Theorem~\ref{Mub}.

\section{Application 1: The distinct slopes in the grid problem}
\label{sec: grid}
In this section we prove Theorem~\ref{pointmain}. To do so, we first collect several facts about the grid.
\subsection{Observations about the grid}
The elements from the grid $[n]^2$ are referred to as \emph{grid points}. The slope of a line $\ell$ in $\mathbb{R}^2$ can be written in the 
form $\frac{u(\ell)}{v(\ell)}$, where $\gcd(|u(\ell)|,|v(\ell)|) = 1$ and $v(\ell) > 0$. Define 
$$ H(l):= \max\{|u(\ell)|,|v(\ell)|\}. $$
The number of grid points on a line with $H(\ell)=s$ is at most $O(n/s)$. The \emph{origin} in the grid is the point $o=(1,1)$. Two grid points $(x,y)$, $(x',y')$ are \emph{mutually visible} if the line segment joining them contains no further grid point. A grid point $(x,y)$ is called \emph{visible} if $(x,y)$ and the origin are mutually visible. A classical result in the theory of visible lattice points says that the number of visible grid points is $\frac{6}{\pi^2} n^2 (1+o(1))$, see e.g. \cite{herzog,Rademacher}. We will use the following crude estimate on the number of visible lattice in a specific subset of the grid repeatedly. For a positive integer $i$, let $V_i$ be the set of visible points in $[2^i]^2\setminus [2^{i-1}]^2$. 
\begin{obs}
\label{vis}
For every $i$ sufficiently large, $|V_i|\geq \frac{1}{4} 2^{2i}$.
\end{obs}
\begin{proof}
The number of visible points in $[2^i]^2$ is $\frac{6}{\pi^2} 2^{2i} (1+o(1))$ and the total number of points in $[2^{i-1}]$ is $\frac{1}{4} 2^{2i}$, implying this observation. 
\end{proof}

 Denote by $\mathcal{H}$ the rank $4$-graph with vertex set $[n]^2$, where the $3$-edges are the collinear triples and 4-edges are trapezoids. Then $g(n)=\alpha(\mathcal{H})$. We remark that Theorem~\ref{LPspar} cannot be applied directly to $\mathcal{H}$ because the codegree condition \eqref{codegreecon} on $\Delta_{2,4}$ is not satisfied. To circumvent this obstruction, we modify the hypergraph $\mathcal{H}$.

 \begin{defn}
 Let $s^*$ be a positive integer. Denote by $\mathcal{H}_{s^*}$ the rank $4$-graph with vertex set $[n]^2$ with edges: 
 \begin{itemize}
 \item $2$-sets of grid points lying on a line $\ell$ which satisfies $H(\ell)\leq s^*$,
 \item $3$-sets of grid points forming collinear triples,
 \item $4$-sets of grid points which form a trapezoid with the additional property that for each choice of two parallel lines $\ell_1,\ell_2$, going through two of the four points each, they satisfy $H(l_i)\geq s^*$. (Note that for a rhombus there are two ways of choosing two parallel lines.)
 \end{itemize}

\end{defn}
The following lemma provides several properties of the auxiliary hypergraph $\mathcal{H}_{s^*}$.
\begin{lemma}
\label{grid}
Let $\varepsilon>0$ and $s^*$ be a positive integer, $\Omega(n^{\varepsilon})= s^*= O(n^{1-\varepsilon})$. Then,

\vspace{0.3cm}
\begin{minipage}[t]{.49\textwidth}
    \begin{itemize}
\item[a)] $\Delta_2(\mathcal{H}_{s^*})=\Theta(n s^*)$,
\item[c)] $\Delta_4(\mathcal{H}_{s^*})=\Theta(n^4 \log n)$,
\item[e)] $\Delta_{3,4}(\mathcal{H}_{s^*})=O\left(\frac{n}{s^*}\right)$,
\end{itemize}
\end{minipage}
\begin{minipage}[t]{.49\textwidth}
    \begin{itemize}
\item[b)] $\Delta_3(\mathcal{H}_{s^*})=\Theta(n^2 \log n)$,
\item[d)] $\Delta_{2,3}(\mathcal{H}_{s^*})= O(n)$,
\item[f)] $\Delta_{2,4}(\mathcal{H}_{s^*})=O\left(\frac{n^3}{s^*}\right)$.
\end{itemize}
\end{minipage}
\end{lemma}

\begin{proof}[Proof of Lemma~\ref{grid} a)] 
Let $i_0$ be sufficiently large such that Observation~\ref{vis} holds for $i\geq i_0$. Let $ i_0 \leq i\leq \lfloor \log(s^*) \rfloor$. Given a visible point $p\in V_i$, let $\ell_p$ be the line through the origin $o$ and $p$. The line $\ell_p$ contains $\Omega(n/2^{i}) $ grid points. For any grid point $q$, the pair $qo$ forms a 2-edge in $\mathcal{H}_{s^*}$ if $q$ is on a line $\ell_p$ for some visible point $p\in [s^*]^2$. Thus, 
$$ \Delta_2(\mathcal{H}_{s^*})\geq \deg_2(o) \geq  \sum_{i=i_0}^{\lfloor \log s^* \rfloor } |V_i| \ \Omega\left(\frac{n}{2^i}\right) 
.$$
By Observation~\ref{vis}, $|V_i|\geq 2^{2i-2}$ and therefore $\Delta_2(\mathcal{H}_{s^*})=\Omega(n s^*)$.

For the upper bound, let $p\in [n]^2$ be an arbitrary grid point. Given $s\leq s^*$, there are at most $4s$ choices for slopes of lines $\ell$ passing through $p$ and satisfying $H(\ell)=s$. On each such line there are at most $O(n/s)$ grid points. Thus, 
$$
\deg_2(p) \leq \sum_{s=1}^{ s^*   } 4s \ O\left(\frac{n}{s}\right)=O(ns^*).
$$
Since the grid point $p$ was chosen arbitrarily, $\Delta_2(\mathcal{H}_{s^*})=O(n s^*)$.
\end{proof}

\begin{proof}[Proof of Lemma~\ref{grid} b)]
We will provide a lower bound on $\Delta_3(\mathcal{H}_{s^*})$ by counting the number of collinear triples containing the origin $o$.
Let $i_0$ be sufficiently large such that Observation~\ref{vis} holds for $i\geq i_0$. Let $ i_0 \leq i\leq \lfloor \log(n) \rfloor$. Given a visible point $p\in V_i$, let $\ell_p$ be the line through the origin $o$ and $p$. The line $\ell_p$ contains $\Omega(n/2^{i}) $ grid points. For grid points $q$ and $q'$ a $3$-set $qq'o$ forms a 3-edge in $\mathcal{H}_{s^*}$ iff $q$ and $q'$ are on a line $\ell_p$ for some visible point $p\in [n]^2$,
we obtain 
$$ \Delta_3(\mathcal{H}_{s^*})\geq \deg_3(o) \geq  \sum_{i=i_0}^{\lfloor \log n \rfloor } |V_i| \ \Omega \left(\binom{\frac{n}{2^{i}}}{2} \right).$$
Using Observation~\ref{vis}, we get $ \Delta_3(\mathcal{H}_{s^*})=\Omega(n^2 \log n)$. 

For the upper bound on $\Delta_3(\mathcal{H}_{s^*})$, let $p\in [n]^2$ be an arbitrary grid point. Given $s\leq s^*$, there are at most $4s$ choices for slopes of lines $\ell$ passing through $p$ and satisfying $H(\ell)=s$. On each such line there are at most $O(n/s)$ grid points. Thus, 
$$
\deg_3(p) \leq \sum_{s=1}^{ s^*  } 4s \ O\left(\left(\frac{n}{s}\right)^2\right)=O(n^2 \log n).
$$
Since $p$ was chosen arbitrarily from $[n]^2$, $\Delta_3(\mathcal{H}_{s^*})= O(n^2 \log n)$. 
\end{proof}

\begin{proof}[Proof of Lemma~\ref{grid} c)]
For the lower bound on $\Delta_4(\mathcal{H}_{s^*})$, we count the number of $4$-edges containing the origin. Let $i\in \mathbb{N}$ satisfying that $ \log s^* \leq i\leq  \log n $. Denote by $T_i$ the number of trapezoids containing the origin $o$ and exactly one pair of parallel sides (i.e. rhombi are not being counted) such that the line through the origin whose slope appears twice in the trapezoid contains a visible point $p\in V_i$. Fixing a point $p\in V_i$, there are $\Omega(n/2^i)$ points on the line $\ell_p$ through the origin $o$ and $p$. For the third point of the trapezoid we choose an arbitrary grid point from $\{(p_1,p_2): n/4\leq p_i \leq 3n/4\}$ which is not on the line $\ell_p$. There are $\Omega(n/2^i)$ ways to choose a forth point to complete a trapezoid, and thus in total
$$
\Delta_4(\mathcal{H}_{s^*})\geq \deg_4(o)\geq \sum_{i= \lceil \log s^* \rceil }^{\lfloor \log n \rfloor}  |V_i| \ \Omega\left( \frac{n}{2^i} \cdot n^2 \cdot \frac{n}{2^i} \right).
$$
Using Observation~\ref{vis}, we get $ \Delta_4(\mathcal{H}_{s^*})=\Omega(n^4 \log n)$.

Now, let $p\in [n]^2$ be an arbitrary point. We will count the number of trapezoids containing $p$. For a given positive integer $s$, let $T_s$ be the number of trapezoids containing $p$ and two parallel sides on lines $\ell_1,\ell_2$ satisfying $H(\ell_1)=H(\ell_2)=s$. Then, clearly
$$
\deg_4(p)\leq \sum_{s=1}^n T_s.
$$
There are at most $4s$ lines $\ell$ through $p$ satisfying $H(\ell)=s$. For each of them, there are at most $O(n/s)$ grid points on it. For the third point of a trapezoid there are at most $n^2$ choices trivially. The last point of the trapezoid needs to be on a line $\ell'$ through the third point and parallel to $\ell$, that are at most $O(n/s)$ choices for such a point. Thus, 
$$T_s\leq O\left( 4s \cdot \frac{n}{s} \cdot n^2 \cdot \frac{n}{s}\right)= O\left(\frac{n^4}{s}\right)$$
which implies
$$
\deg_4(p)\leq \sum_{s=1}^n T_s  =O(n^4 \log n).
$$
Since $p$ was arbitrary, we conclude $\Delta_{4}(\mathcal{H}_{s^*})= O(n^4 \log n)$.
\end{proof}

\begin{proof}[Proof of Lemma~\ref{grid} d)]

Since any line in $\mathbb{R}^2$ contains at most $n$ grid points, we immediately obtain $\Delta_{2,3}(\mathcal{H}_{s^*})= O(n)$. 
\end{proof}

\begin{proof}[Proof of Lemma~\ref{grid} e)]

Let $A=\{p_1,p_2,p_3\}$ be a $3$-element subset of the grid. If the points $p_1,p_2,p_3$ form a collinear triple, then they cannot be contained in a trapezoid and thus $\deg_4(A)=0$. Otherwise, for $1\leq i < j \leq 3$, let $\ell_{i,j}$ be the line through $p_i$ and $p_j$. A grid point $p$ forming a trapezoid with the points from $A$, must lie on the line thorough $p_m$ parallel to $\ell_{i,j}$ for some choice $\{i,j,m\}=[3]$. For any such point $p$, the set $\{p,p_1,p_2,p_3\}$ is an edge in $\mathcal{H}_{s^*}$, if additionally $H(\ell_{i,j})\geq s^*$. Thus, there are at most $O(n/s^*)$ choices for $p$. We conclude $\Delta_{3,4}(\mathcal{H}_{s^*})=O(n/s^*)$.
\end{proof}

\begin{proof}[Proof of Lemma~\ref{grid} f)]
By Lemma~\ref{grid} e),
\begin{equation*}\Delta_{2,4}(\mathcal{H}_{s^*}) \leq n^2 \Delta_{3,4}(\mathcal{H}_{s^*})=O\left(\frac{n^3}{s^*}\right). \qedhere
\end{equation*}
\end{proof}

\subsection{Proof of Theorem~\ref{pointmain}}
\label{prooflb}
Let $s^*=\lceil n^{1/3} \rceil$ and $f=\log^{1/2}n$. Then, by Lemma~\ref{grid}
\begin{gather*}
 \gamma:=\max_{2\leq\ell\leq 4}\Delta_{\ell}(\mathcal{H}_{s^*})^{1/(\ell-1)}= \Theta(\max \{ n^{4/3},n \log^{1/2}n, n^{4/3} \log^{1/3}n  \})
 =\Theta(n^{4/3} \log^{1/3}n).
\end{gather*}
Further,
\begin{gather*}
\Delta_{2,3}(\mathcal{H}_{s^*})= O(n), \quad \quad
\Delta_{3,4}(\mathcal{H}_{s^*})= O(n^{2/3}), \quad \quad \text{and} \quad \quad
\Delta_{2,4}(\mathcal{H}_{s^*})= O(n^{8/3}).
\end{gather*}
Therefore,
\begin{gather*}
\Delta_{2,3}(\mathcal{H}_{s^*})= O(n) \leq  \Theta\left( \frac{n^{4/3} \log^{1/3}n}{f} \right)  = \frac{\gamma}{f},\\
\Delta_{3,4}(\mathcal{H}_{s^*})=O(n^{2/3}) \leq \Theta \left( \frac{n^{4/3} \log^{1/3}n}{f} \right) = \frac{\gamma}{f},\\
\Delta_{2,4}(\mathcal{H}_{s^*})=  O(n^{8/3}) \leq \Theta\left(\frac{n^{8/3} \log^{2/3}n}{f} \right)  = \frac{\gamma^{2}}{f}.
\end{gather*}
We conclude that the conditions \eqref{codegreecon} are satisfied. Further, for the graph triangle $K_3$, by Lemma~\ref{grid} 
$$\Delta_{K_3}(\mathcal{H}_{s^*}) \leq \Delta_{2}(\mathcal{H}_{s^*})^2 = O(n^{8/3}) \leq \Theta\left(\frac{n^{8/3} \log^{2/3}n}{f} \right)  = \frac{\gamma^{2}}{f}.
$$
We conclude that \eqref{codegreeconK3} is satisfied. Note that
$$
\max \left\{\left(  \frac{\Delta_{\ell}}{\log \log n} \right)^{\frac{1}{\ell-1}} \right\}=  \max \left\{\left(  \frac{n^{4/3}}{\log \log n} \right), \left(  \frac{n^2 \log n}{\log \log n} \right)^{\frac{1}{2}} , \left(  \frac{n^4 \log n}{\log \log n} \right)^{\frac{1}{3}} \right\}= \Theta \left(  \frac{n^{4/3} \log^{1/3} n}{(\log \log  n)^{1/3}} \right).
$$
Therefore, by Theorem~\ref{LPspar},
$$
\chi(\mathcal{H}_{s^*})= O\left(\max_{2\leq \ell \leq 4} \left\{\left(  \frac{\Delta_\ell}{\log f} \right)^{\frac{1}{\ell-1}} \right\}\right)=O \left(  \frac{n^{4/3} \log^{1/3} n}{(\log \log  n)^{1/3}} \right).
$$
Using that every independent set in $\mathcal{H}_{s^*}$ is an independent set in $\mathcal{H}$, we conclude
$$
g(n)=\alpha(\mathcal{H})\geq \alpha(\mathcal{H}_{s^*})=\Omega \left(  \frac{n^{2/3} (\log \log  n)^{1/3}}{ \log^{1/3}n} \right).
$$
This completes the proof of Theorem~\ref{pointmain}.

\section{\texorpdfstring{Application 2: The Tur\'an density of $H_3^r$}{TEXT}}
\label{Turan}
In this section we prove Theorem~\ref{greedyK3} using Theorem~\ref{Mub}. 
\subsection{The blow-up operation}
\label{sec: 1}
Let $G$ be an $r$-graph. The \emph{blow-up} $G(t)$ of $G$ is the $r$-graph which is obtained from $G$ by replacing each vertex $x\in V(G)$ by $t$ vertices $x^1,\ldots,x^t$ and each edge $x_1 \cdots x_r\in E(G)$ by $t^r$ edges $x_1^{a_1}\cdots x_r^{a_r}$ with $1 \leq  a_1,\ldots, a_r \leq t$. We say an $r$-graph $H$ has \emph{complete shadow} if every pair of vertices is contained in an edge. The following two observations about blow-ups are essential. 

\begin{lemma}
\label{blowHfree}
Let $m\in \mathbb{N}$ and $H$ be an $r$-graph with complete shadow. If $G$ is an $H$-free $r$-graph, then $G(m)$ is $H$-free.   
\end{lemma}
\begin{proof}
Towards contradiction, assume that $G(m)$ contains a copy of $H$ on some vertex set \\
$U\subseteq V(G(m))$. Then, because $H$ has complete shadow, for any two vertices $x_1^{i},x_2^j\in U$ we have $x_1\neq x_2$. Therefore $G$ contains a copy of $H$ too, a contradiction. 
\end{proof}
\begin{lemma}
\label{blowHdensity}
Let $G$ be an $r$-graph with $m$ vertices, where $m\geq r^2$. Let $n$ be an integer divisible by $m$, and let $t=n/m$. Then $$
\frac{|E(G(t))|}{ \binom{n}{r}} \geq \frac{1}{4} \frac{|E(G)|}{ \binom{m}{r}}. $$
\end{lemma}
\begin{proof}
We have
\begin{align*}
\hspace{1cm}
\frac{|E(G(t))|}{\binom{n}{r}} &= \frac{|E(G)| t^r}{ \binom{n}{r}}=  \frac{|E(G)|t^r}{\binom{m}{r}\frac{n!(m-r)!}{m!(n-r)!}}= \frac{|E(G)|}{\binom{m}{r}} \frac{m!(n-r)!}{n!(m-r)!} \frac{n^r}{m^r}\geq \frac{|E(G)|}{\binom{m}{r}} \frac{m!}{(m-r)!m^r}\\
&\geq \frac{|E(G)|}{\binom{m}{r}} \left(1-\frac{r}{m}\right)^r \geq \frac{|E(G)|}{\binom{m}{r}} \left(1-\frac{1}{r}\right)^r\geq \frac{1}{4}  \frac{|E(G)|}{\binom{m}{r}}.
\hspace{4.65cm}
\qedhere
\end{align*}
\end{proof}

Informally, the takeaway of Lemmas~\ref{blowHfree} and \ref{blowHdensity} is that in order to understand the asymptotic behavior of Tur\'an density of $H_3^r$ it suffices to consider host graphs with roughly $r^2$ vertices.

\subsection{Proof of Theorem~\ref{greedyK3}}
Let $G$ be the auxiliary $3$-graph with vertex set $V=\binom{[r^2]}{r}$ and edges $A\in E(G)$ iff $A$ forms a copy of $H_3^r$ in $K_{r^2}^r$. Note that 
\begin{align*}
\Delta=\Delta(G)= \binom{r}{r-2}(r^2-r),
\end{align*}
because simply every edge in $K_{r^2}^r$ is contained in this many copies of $H_3^r$. Therefore, $r^4/3 \leq \Delta \leq r^{4}$ for $r$ sufficiently large. Set $f=\Delta^{1/10}$ and let $A=\{e_1,e_2\}\subseteq V(G)$ of size $2$. We will upper bound $\deg_3(A)$. If $|e_1\cap e_2|<r-1$, then $\deg_3(A)=0$. If  $|e_1\cap e_2|=r-1$, then any $f\in V(G)$ satisfying that $e_1,e_2,f$ form a copy of $H_3^r$ has the property that $f\subseteq e_1\cup e_2$. Since $f\neq e_1$ and $f\neq e_2$, we have 
$$
\deg_3(A)\leq \binom{|e_1\cup e_2|}{r}-2 = \binom{r+1}{r}-2 =r-1.
$$
We conclude $\Delta_{2,3}(G)\leq r-1$ and thus condition \eqref{condMub} holds for $G$:
\begin{align*}
\Delta_{2,3}(G) \leq r-1 \leq \left(\frac{r^4}{3}\right)^{\frac{1}{3}}\leq \Delta^{\frac{1}{3}}\leq \frac{\Delta^{\frac{1}{2}}}{\Delta^{\frac{1}{10}}}=\frac{\Delta^{\frac{1}{2}}}{f}.
\end{align*}
By Theorem~\ref{Mub} the $3$-graph $G$ satisfies 
$$\chi(G)=O\left(\left(\frac{\Delta}{\log  \Delta}\right)^{\frac{1}{2}} \right)= O\left(\frac{r^2}{(\log r)^{\
\frac{1}{2}}}\right).$$
Therefore $G$ contains an independent set of size at least 
$$\alpha(G)=\Omega\left(\binom{r^2}{r}\frac{(\log r)^{\
\frac{1}{2}}}{r^2} \right).$$
This independent set corresponds to an $H_3^r$-free $r$-graph $H$ on $r^2$ vertices with $\alpha(G)$ many edges.  Let $t=n/r^2$. By considering the blow-up $H(t)$ of this $r$-graph and Lemmas~\ref{blowHfree} and \ref{blowHdensity} we obtain 
\begin{equation*}
\pi(H_3^r)\geq \limsup_{n \rightarrow \infty} \frac{|E(H(t))|}{ \binom{n}{r}} \geq \frac{1}{4} \frac{|E(H)|}{ \binom{r^2}{r}}=\Omega\left(\frac{(\log r)^{\
\frac{1}{2}}}{r^2} \right),
 \end{equation*}
completing the proof of Theorem~\ref{greedyK3}.

\section{Acknowledgements}
The author thanks Maria Axenovich, J\'ozsef Balogh, Alberto Espuny D\'iaz, Peter Kaiser, Lina Li, Dingyuan Liu, Let\'icia Mattos, Ethan Patrick White for valuable discussions on the topic of this paper. Special thanks are extended to Hong Liu and Zixiang Xu;
Hong Liu for organizing the 1st ECOPRO combinatorial week workshop in Daejeon, and Zixiang Xu for presenting the distinct slopes problem.

\noindent

\end{document}